\documentclass[12pt]{amsart}
\usepackage{mathrsfs} 
\usepackage[sorted-cites]{amsrefs}
\usepackage[matrix,arrow,tips,curve]{xy}
\usepackage[english]{babel}
\usepackage{amsmath}
\usepackage{amssymb}
\numberwithin{equation}{section}
\theoremstyle{plain}
\usepackage{enumerate}
\usepackage{graphicx}
\usepackage{hyperref}
\usepackage{color}
\DeclareMathAlphabet{\pazocal}{OMS}{zplm}{m}{n}

\usepackage{xcolor}

\newtheorem{theorem}{Theorem} [section]
\newtheorem{lemma}{Lemma}[section]
\newtheorem{proposition}{Proposition}[section]

\newtheorem{cor}{Corollary}[section]

\theoremstyle{remark}
\newtheorem{remark}{Remark}[section]

\renewcommand{\bar}{\overline}

\hypersetup{pdfpagemode=UseNone} 
\hypersetup{pdfstartview=FitH}

\title[ eigenvalue estimate for self-shrinkers ]{An eigenvalue estimate for  self-shrinkers in a Ricci Shirinker}

\thanks{F. Conrado was partially supported by CNPq/Brazil [Grant: 408834/2023-4].}
\thanks{D. Zhou was partially supported by CNPq/Brazil [Grant:  308067/2023-1] and FAPERJ/ Brazil [Grant: E-26/200.386/2023].}

\address{Departamento de Matem\'atica, Universidade Federal de Sergipe, Jardim Rosa Elze, S\~ao Crist\'ov\~ao, SE  49100-000, Brazil}
\author[Franciele Conrado]{Franciele Conrado}

\email{franciele@mat.ufs.br}
\author[Detang Zhou]{Detang Zhou}
\address{Departamento de Geometria, Instituto de Matem\'atica e Estat\'istica, Universidade Federal Fluminense, S\~ao Domingos,
Niter\'oi, RJ 24210-201, Brazil}
\email{zhoud@id.uff.br}

\begin{document}
\maketitle

\begin{abstract}
In this paper, we study the drifted Laplacian $\Delta_f$ on a hypersurface $M$ in a Ricci shrinker $(\overline{M},g,f)$.  We prove that the spectrum of $\Delta_f$ is discrete for immersed hypersurfaces with bounded weighted mean curvature in a Ricci shrinker with a mild condition on the potential function. Next, we give a lower bound for the first nonzero eigenvalue of $\Delta_f$ when the  hypersurface is  an embedded $f$-minimal one.  This estimate contains the case of compact minimal hypersurfaces in a positive Einstein manifold, in particular Choi and Wang's estimate for minimal hypersurfaces in a round sphere. The estimate also recovers the ones of Ding-Xin and Brendle-Tsiamis on self-shrinkers. 
\end{abstract}

\section{Introduction}

A \textit{complete smooth metric measure space} $(N, g, f)$ is a complete Riemannian manifold $(N,g)$ accompanied by a smooth function $f$ defined on $N$. The \textit{drifted Laplacian} of a complete smooth metric measure space $(N,g,f)$ is defined as:
$$\Delta_f=\Delta-\langle \nabla f, \nabla \cdot \rangle$$
\noindent where $\Delta$ and $\nabla$ are the Laplacian and gradient on $(N,g)$, respectively. The second-order operator $\Delta_f$ is a self-adjoint operator on the space of square integrable function on $N$ with respect to the measure induced by the weighted volume element $e^{-f}dv$.
The \textit{Bakry-\'Emery Ricci curvature} of a complete smooth metric measure space $(N,g,f)$ is defined as:
$$\textrm{Ric}_f=\textrm{Ric}+\nabla^2f$$
\noindent where $\nabla^2 f$ represents the Hessian of $f$ and $\textrm{Ric}$ represents the Ricci curvature of $(N,g)$. Recall that a complete smooth metric measure space $(N,g,f)$ is  a \textit{shrinking gradient Ricci soliton} or simply a \textit{Ricci shrinker} if:
\begin{equation}\label{RS0}\textrm{Ric}_f=k g\end{equation}
\noindent for some positive constant $k$. In this case, the function $f$ is called a \textit{potential function} of the Ricci shrinker. By scaling $g$, one can normalize $k=\frac{1}{2}$ so that
\begin{equation}\label{RS}
\textrm{Ric}_f=\frac{1}{2} g.
\end{equation}

Note that any Einstein manifold with positive scalar curvature accompanied by a constant function is a Ricci shrinker. An important example of a Ricci shrinker is the \textit{Gaussian shrinker} $(\mathbb{R}^n, g_{can}, \frac{|x|^2}{4})$, where $g_{can}$ represents the Euclidean metric, and the potential function is given by $f(x) = \frac{|x|^2}{4}$. Another example of a Ricci shrinker is the \textit{cylinder Ricci shrinker} $(\mathbb{S}_R^{n-k} \times \mathbb{R}^k, g_c, f_c)$ with the product metric $g_c$ and the potential function $f_c(x,y) = \frac{|y|^2}{4}$, where $(x,y) \in \mathbb{S}_R^{n-k}\times \mathbb{R}^k$. Here, $\mathbb{S}_R^{n-k}$ denotes the $(n-k)$-dimensional round sphere of radius $R= \sqrt{2(n-k-1)}$, with $k \geq 1$ and $n-k \geq 2$. We refer to Cao's paper \cite{Cao2010} for the background and \cite{CZ2023, MW2015, MW2017} for some recent results.

Since Ricci shrinkers lie at the intersection of critical metrics and geometric flows, and have been extensively studied, it is natural and important to study submanifolds in Ricci shrinkers and, more generally, in complete smooth metric measure space. Let $(\overline{M},g,f)$ be a complete smooth metric measure space and $M$ be a submanifold in $\overline{M}$. The function $f$ restricted to $M$ is also denoted by $f$. The {\it weighted mean curvature vector} $\vec{H}_f$ of $M$ is defined by
 $$\vec{H}_f=\vec{H}+(\overline{\nabla} f)^{\perp}$$
\noindent where $\vec{H}$ is the mean curvature vector of $M$, $\overline{\nabla} f$ denotes the gradient of $f$ on $\overline{M}$ and $\perp$ denotes the projection onto the normal bundle of $M$. In the case of hypersurfaces, the weighted mean curvature $H_f$ of $M$ is defined as $\vec{H}_f=-H_f\eta$, where $\eta$ is the unit normal field on $M$. The submanifold $M$ is called {\it $f$-minimal} if its weighted mean curvature vector $\vec{H}_f$ vanishes identically. Note that when $f$ is a constant function, an $f$-minimal submanifold is just a minimal submanifold. Important and interesting examples of $f$-minimal submanifolds are self-shrinkers, self-expanders, and translating solitons for mean curvature flows in Euclidean spaces with different potential functions. Recall that a self-shrinker is a submanifold of $\mathbb{R}^{m}$ satisfying
$$\vec{H}+\frac{1}{2}x^{\perp}=0$$
\noindent where $x$ is the position vector in $\mathbb{R}^{m}$. Hence, a self-shrinker is an $f$-minimal submanifold in the Gaussian shrinker $(\mathbb{R}^{m},g_{can},f)$ with $f(x)=\frac{|x|^2}{4}$. In this case, it is usual to denote the drifted Laplacian $\Delta_f$ of $M$  by $\mathcal{L}$. Self-shrinkers play a key role in the study of type I singularities of the mean curvature flow. For more information on self-shrinkers, see \cite{CM} and its references.

In this paper, we are interested in studying the discreteness of the spectrum and lower estimates for the first eigenvalue of the drifted Laplacian $\Delta_f$ of an embedded $f$-minimal hypersurface in a complete smooth metric measure space $(\overline{M},g,f)$.

Choi and Wang \cite{CW1983} proved a lower bound for the first nonzero eigenvalue of the Laplacian on a closed embedded minimal hypersurface in a simply connected complete Riemannian manifold with Ricci curvature bounded below by a positive constant. More precisely, they proved the following result.

\begin{theorem}\label{t1} Let $\overline{M}$ be a simply connected complete Riemannian manifold with Ricci curvature bounded below by a constant $k>0$  and $M$ be a closed embedded minimal hypersurfaces. Then, the first nonzero eigenvalue of the Laplacian $\Delta$ on $M$ is at least $\frac{k}{2}$.
\end{theorem}

It is very important to highlight that this result is related to Yau's conjecture. Let $\mathbb{S}^{m+1}$ be the unit sphere in $\mathbb{R}^{m+2}$ and $M$ a closed embedded minimal hypersurface within $\mathbb{S}^{m+1}$. It is well known that all coordinate functions on $\mathbb{R}^{m+2}$ restricted to $M$ are eigenfunctions of the Laplacian $\Delta$ on $M$ corresponding to the eigenvalue $m$. In particular, the first eigenvalue $\lambda_1(\Delta)$ of $\Delta$ satisfies $\lambda_1(\Delta)\leq m$. Yau conjectured \cite{YAU} that $\lambda_1(\Delta)=m$. Observe that, by Theorem \ref{t1} we have $\lambda_1(\Delta)\geq \frac{m}{2}$, favoring Yau's conjecture. Recently, this estimate was improved by Jiménez, Chinchay and Zhou \cite{JCZ}, who showed that $\lambda_1(\Delta)>\frac{m}{2}+C$, where $C$ is a positive constant that depends only on $m$ and the norm of the second fundamental form of $M$. For more progress on Yau's conjecture, we refer to \cite{CS2009}, \cite{DSS}, and \cite{TXY}.

Later, Choi and Schoen \cite{CS1985}, used a covering argument to show that the hypothesis of $\overline{M}$ being simply connected in Theorem \ref{t1} is not necessary. Ma and Du \cite{MD} extended Theorem \ref{t1} to the first eigenvalue of the drifted Laplacian $\Delta_f$ on a closed embedded $f$-minimal hypersurface in a simply connected compact complete smooth metric measure space $(\overline{M},g,f)$ with Bakry-Émery Ricci curvature $\overline{\textrm{Ric}}_f$ bounded below by a positive constant. Recently, Li and Wei \cite{LW} also used a covering argument to show that the hypothesis of $\overline{M}$ being simply connected is not necessary in the result of Ma and Du. Therefore, the following result was established.

\begin{theorem}\label{t2} Let $(\overline{M},g,f)$ be a compact complete smooth metric measure space with Bakry-Émery Ricci curvature $\overline{\textrm{Ric}}_f$ bounded below by a constant $k>0$ and $M$ be a closed embedded $f$-minimal hypersurface. Then, the first nonzero eigenvalue of $\Delta_f$ on $M$ is at least $\frac{k}{2}$.
\end{theorem}

Note that, by the Bonnet-Myers Theorem, a complete Riemannian manifold with Ricci curvature bounded below by a positive constant is compact. However, a complete smooth metric measure space with Bakry-Émery Ricci curvature bounded below by a positive constant is not necessarily compact. An example is the Gaussian shrinker $(\mathbb{R}^n,g_{can},\frac{|x|^2}{4})$. Hence, Theorem \ref{t2} cannot be applied to self-shrinkers. Ding and Xin \cite{DX} proved that Theorem \ref{t2} is valid when the ambient smooth metric space measure is the Gaussian shrinker $(\mathbb{R}^n,g_{can},\frac{|x|^2}{4})$, that is, if $M$ is a closed $n$-dimensional embedded self-shrinker in Euclidean space $\mathbb{R}^{n+1}$, then the first nonzero eigenvalue of the drifted Laplacian $\mathcal{L}=\Delta-\frac{1}{2}\langle x, \nabla \cdot\rangle$ on $M$ is at least $\frac{1}{4}$.

Recently, Cheng, Mejia, and Zhou \cite{CMZ2014} proved, with a hypothesis on the potential function $f$, that Theorem \ref{t2} holds when the ambient complete smooth metric measure space is noncompact. More precisely, they proved the following theorem.

\begin{theorem}\label{t3} Let $(\overline{M},g,f)$ be a complete smooth metric measure space with Bakry-Émery Ricci curvature $\overline{\textrm{Ric}}_f$ bounded below by a constant $k>0$ and $M$ be a closed embedded $f$-minimal hypersurfaces. If there is a bounded domain $D$ in $\overline{M}$ with convex boundary $\partial D$ so that $M$ is contained in $D$, then the first nonzero eigenvalue of $\Delta_f$ on $M$ is at least $\frac{k}{2}$.
\end{theorem}

A closed $n$-dimensional embedded self-shrinker in Euclidean space $\mathbb{R}^{n+1}$ satisfies the assumption of Theorem \ref{t3}. Hence, Theorem \ref{t3} implies the result of Ding and Xin.
 
Note that all the results mentioned above are for closed embedded $f$-minimal hypersurfaces, which guarantees the discreteness of the spectrum of the drifted Laplacian $\Delta_f$. The noncompactness of the hypersurface allows the possibility that the spectrum of $\Delta_f$ is not discrete. Cheng and Zhou \cite{XCDZ} proved that this cannot happen for properly immersed self-shrinkers. Later, Vieira and Zhou \cite{VZ2018} extended this result for properly immersed $f$-minimal submanifolds in a Ricci shrinker with convex potential function. 

As many examples of self-shrinkers are not compact, it is natural to extend the Theorem of Ding and Xin to properly embedded $n$-dimensional self-shrinkers in $\mathbb{R}^{n+1}$.  Brendle and Tsiamis \cite{BT2024} proved some integral inequalities without knowing the existence of the eigenfunctions. Combining with the discreteness result of Cheng and Zhou\cite{XCDZ} they actually proved the following.

\begin{theorem}\label{t4} If $M$ is a properly embedded $n$-dimensional self-shrinker in Euclidean space $\mathbb{R}^{n+1}$, then the first nonzero eigenvalue of $\mathcal{L}$ on $M$ is at least $\frac{1}{4}$.
\end{theorem}

A very important class of hypersurfaces in a complete smooth metric measure space is the class of {\it hypersurfaces with constant weighted mean curvature} (CWMC). The class of $f$-minimal hypersurfaces is the particular case of CWMC hypersurfaces with weight mean curvature  vanishes identically. An example of CWMC hypersurface is $\mathbb{S}_r^{n-k}\times \mathbb{R}^k$ in Gaussian shrinker $(\mathbb{R}^{n+1},g_{can},\frac{|x|^2}{4})$. Another example is  $\mathbb{S}_R^{n-k}\times \mathbb{S}_r^k$ in cylinder Ricci shrinker $(\mathbb{S}_R^{n-k} \times \mathbb{R}^{k+1}, g_c, f_c)$. In the case that the ambient complete smooth metric measure space is the Gaussian shrinker, a hypersurface  with constant weight mean curvature $\lambda$ is also called of {\it $\lambda$-hypersurface}. Note that when $\lambda=0$, a $\lambda$-hypersurface is just a self-shrinker. It is not known about the discreteness of the spectrum of  the drifted Laplacian on a $\lambda$-hypersurface (not necessarily self-shrinker), more generally, about the discreteness of the spectrum of the drifted Laplacian on a CWMC hypersurface (not necessarily $f$-minimal).

In this paper, we prove two main results. The first guarantees the discreteness of the spectrum of the drifted Laplacian on properly immersed submanifold in a Ricci shrinker with a hypothesis on the potential function and an upper bound on the norm of the weighted mean curvature vector, see Theorem \ref{tpdis} for more precise statement. We prove this theorem by adapting the argument used in the proof of the discreteness result of Viera and Zhou in \cite{VZ2018}. As one special case of Theorem \ref{tpdis}, we have the following result that guarantees the discreteness of the drifted Laplacian spectrum on a properly immersed CWMC hypersurface (not necessarily $f$-minimal) in a Ricci shrinker with a hypothesis on the potential function. 

\begin{theorem}\label{p3} Let $M$ be a CWMC properly immersed hypersurface in a Ricci shrinker $(\overline{M},g,f)$ with $\overline{Ric}_f=kg$.  If there is a constant $0\leq\epsilon<2k$ such that the trace of Hessian of $f$ on $\overline{M}$ restricted on the normal bundle of $M$ satisfies $\overline{R}+2\textrm{tr}_{M^{\perp}} \overline{\nabla}^2f\geq - \epsilon  f$ outside a compact set, then the spectrum of $\Delta_f$ with domain $L^2(M,e^{-f}dv)$ is discrete.
\end{theorem}

As an immediate consequence of Theorem \ref{p3}, we have the discreteness of the spectrum of the drifted Laplacian on a properly immersed $\lambda$-hypersurface (not necessarily self-shrinker). More precisely,

\begin{cor}
If $M$ is a properly immersed $\lambda$-hypersurface in $\mathbb{R}^{n+1}$, then the spectrum of $\mathcal{L}$ with domain $L^2(M,e^{-\frac{|x|^2}{4}}dv)$ is discrete.
\end{cor}

The second main result we prove in this paper is to use the discreteness result and eigenfunctions to give a lower bound for the first nonzero eigenvalue of the drifted Laplacian on a properly embedded $f$-minimal hypersurface in a Ricci shrinker with a hypothesis on the potential function. More precisely, we prove the following result.

\begin{theorem}\label{tp} Let $M$ be a $f$-minimal properly embedded hypersurface in a Ricci shrinker $(\overline{M},g,f)$ such that $\overline{\textrm{Ric}}_f=k g$. Assume that there is a constant $0\leq\epsilon<2k$ such that the trace of Hessian of $f$ on $\overline{M}$ restricted on the normal bundle of $M$ satisfies $\overline{R}+2\textrm{tr}_{M^{\perp}} \overline{\nabla}^2f\geq - \epsilon  f$ outside a compact set. Then the spectrum of $\Delta_f$ with domain $L^2(M,e^{-f}dv)$ is discrete and the first nonzero eigenvalue is at least $\frac{k}{2}$.
\end{theorem}

\begin{remark}
The condition on the potential function in Theorem \ref{tp} is not strong, since the scalar curvature is nonnegative and in many examples, such as cylinder shrinkers, their potential functions are convex. In particular, a $n$-dimensional properly embedded self-shrinker in Euclidean space $\mathbb{R}^{n+1}$ satisfies the assumption of Theorem \ref{tp}. Therefore, Theorem \ref{tp} implies Theorem \ref{t4} of Brendle and Tsiamis. Using eigenfunctions, our proof is simpler.  On the other hand, Theorem \ref{tp} also contains the case of compact minimal hypersurfaces in a positive Einstein manifold, in particular Choi and Wang's estimate for minimal hypersurfaces in a round sphere. 
   
\end{remark}
\begin{remark}
We would like to point out that the estimates in the above theorems are not sharp as shown in \cite{JCZ}. Also, the proof does not apply to the case CWMC hypersurfaces.
\end{remark}

The rest of the article is organized as follows. In Section \ref{sec2}, we fix some notation and recall the key definitions and necessary results in our arguments. In Section \ref{sec3}, we prove the discreteness of the spectrum of the drifted Laplacian on properly immersed submanifold in a Ricci shrinker with a hypothesis on the potential function and an upper bound on the norm of the weighted mean curvature vector. In Section \ref{sec4}, we prove an integral inequality of the weighted Bochner formula which we apply in the following section to prove Theorem \ref{tp}. Finally, in Section \ref{sec5}, we present the proof of Theorem \ref{tp}.

\section{Preliminaries}\label{sec2}

In this section, we fix some notation and recall fundamental definitions and necessary results concerning Ricci shrinkers and $f$-minimal submanifolds for a better understanding of the subsequent sections.

Let us recall some useful facts about Ricci shrinkers.

\begin{theorem}[Chen, \cite{chen}]\label{chen} Let $(N,g,f)$ be a Ricci shrinker. Then $g$ has nonnegative scalar curvature $R\geq 0$. 
\end{theorem}

\begin{remark} By using the maximum principle, Pigola, Rimoldi, and Setti  \cite{prs} proved that a Ricci shrinker has positive scalar curvature unless it is isometric to Euclidean space.
\end{remark}

\begin{theorem}[Hamilton, \cite{hamilton}]\label{hamilton} If $(N,g,f)$ is a Ricci shrinker satisfying (\ref{RS0}) then $R+|\nabla f|^2-2kf$ is a constant function.
\end{theorem}

The following theorem gives us a precise estimate of the growth of potential functions of noncompact Ricci shrinkers.

\begin{theorem}[Cao-Zhou, \cite{caozhou}]\label{CZ1} Let $(N,g,f)$ be a noncompact Ricci shrinker satisfying (\ref{RS}) such that $f=|\nabla f|^2+R$. Fix a point $p\in N$ and consider the distance function $r(x)=d(x,p)$, $x\in N$. Then
$$\frac{1}{4}(r(x)-c)^2\leq f(x)\leq \frac{1}{4}(r(x)+c)^2$$
\noindent for every $x\in N$ with $r(x)$ sufficiently large, where $c$ is a positive constant depending only on $n$ and the geometry of $g$ on the unit ball $B_1(p)$.
\end{theorem}

\begin{remark}\label{r1} By the previous theorem $f(x) \to +\infty$ when $r(x)\to +\infty$. In particular, the potential function of a Ricci shrinker attains a global minimum.
\end{remark}

The following theorem tells us that noncompact Ricci shrinkers have at most Euclidean volume growth. This result can be viewed as an analogous of the well-known theorem of Bishop that a complete non-compact Riemannian manifold with non-negative Ricci curvature has at most Euclidean volume growth.

\begin{theorem}[Cao-Zhou, \cite{caozhou}]\label{CZ2}  Let $(N,g,f)$ be a noncompact Ricci shrinker. Then, there exists some positive constant $C>0$ such that
$$\textrm{Vol}(B_p(r))\leq Cr^n$$
\noindent for $r>0$ sufficiently large.
\end{theorem}

\begin{remark}\label{r2} It follows from Theorems \ref{CZ1} and \ref{CZ2} that if $(N,g,f)$ is a Ricci shrinker then the measure induced by the weighted volume element $e^{-f}dv$ is a finite measure, i.e.,
$$\displaystyle\int_N e^{-f}dv<\infty.$$ 
In particular, a Ricci shrinker has finite fundamental group.
\end{remark}

\begin{remark}  The finiteness of weighted volume plays a important role in the study of spectrum of Laplacian and properness of immersed submanifolds \cite{CZ2013}. In \cite{CVZ2021}, it has been generalized to a quite more general cases.

\end{remark}

Let $(\overline{M},g,f)$ be an $m$-dimensional complete smooth metric measure space and $M$ be an $n$-dimensional submanifold in $\overline{M}$ ($n<m$). The function $f$ restricted to $M$ is also denoted by $f$. We denote by $dv$ and $d\sigma$ the volume elements of $(\overline{M},g)$ and $(M,g)$, respectively. Furthermore, we denote by $\nu$ the measure on $\overline{M}$ induced by the weighted volume element $e^{-f}dv$ and by $\mu$ the measure on $M$ induced by the weighted volume element $e^{-f}d\sigma$. In this paper, unless otherwise specified, the notations with a bar denote the quantities on $(\overline{M},g)$, while the notations without a bar denote the quantities on $(M,g)$.

The {\it second fundamental form} $B$ of $M$ is defined by 
$$B(X,Y)=(\overline{\nabla}_XY)^{\perp}, \ \ X,Y\in T_pM, \ \ p\in M,$$
\noindent where $\perp$ denotes the projection onto the normal bundle of $M$. 

The {\it mean curvature vector} $\vec{H}$ of $M$ is defined by
 $$\vec{H}= \textrm{tr} \ B= \sum_{i=1}^n (\overline{\nabla}_{e_i}e_i)^{\perp}$$
 \noindent where $\{e_1,e_2, \cdots, e_n\}$ is a local orthonormal frame on $M$.
 
 The {\it weighted mean curvature vector} $\vec{H}_f$ of $M$ is defined by
 $$\vec{H}_f=\vec{H}+(\overline{\nabla} f)^{\perp}.$$

 In the case of hypersurfaces, the mean curvature $H$ of $M$ is defined as $\vec{H}=-H\eta$,  where $\eta$ is the unit normal field on $M$. In the same way, the weighted mean curvature $H_f$ of $M$ is defined as $\vec{H}_f=-H_f\eta$. 
 
 The submanifold $M$ is called {\it $f$-minimal} if its weighted mean curvature vector $\vec{H}_f$ vanishes identically. 

\begin{remark} It is well known that an $f$-minimal submanifold is a critical point of the weighted volume functional
    $$V_f(S)=\int_S e^{-f}dv$$
\noindent where $S\subset \overline{M}$ is a measurable subset. Furthermore, it is a minimal submanifold in $\overline{M}$ equipped with the conformal metric $\overline{g}=e^{\frac{-2f}{n}}g$. \end{remark}

By a direct calculation, we obtain that for any $\Phi\in C^2(\overline{M})$, the following equations hold:
\begin{equation}\label{eq}
\overline{\Delta}\Phi=\Delta \Phi-\langle \vec{H}_f,(\overline{\nabla}\Phi)^{\perp}\rangle + \langle(\overline{\nabla}f)^{\perp}, (\overline{\nabla}\Phi)^{\perp} \rangle + \textrm{tr}_{M^{\perp}} \overline{\nabla}^2\Phi. 
\end{equation}
\noindent and
\begin{equation}\label{eq1}
\overline{\Delta}_f\Phi=\Delta_f\Phi-\langle\vec{H}_f, \overline{\nabla}\Phi\rangle+ \textrm{tr}_{M^{\perp}} \overline{\nabla}^2\Phi. 
\end{equation}
\noindent where $\textrm{tr}_{M^{\perp}} \overline{\nabla}^2\Phi$ denotes the trace of Hessian of $\Phi$ on $\overline{M}$ restricted on the normal bundle of $M$.

\section{The discreteness of spectrum of $\Delta_f$}\label{sec3}

\medskip
In \cite{XCDZ} it has been proved that the drifted Laplacian on any properly immersed self-shrinker for mean curvature flow has a discrete spectrum. More precisely, 
\begin{proposition}
 Let $M$ be a properly immersed self-shrinker in $\mathbb{R}^{m}$.  Then the drifted Laplacian $\mathcal{L}$  with domain $L^2(M,e^{-\frac{|x|^2}{4}}d\sigma)$ has discrete spectrum.
\end{proposition}

We will show that the discreteness of the drifted Laplacian holds for more general cases. For example,  
\begin{proposition}
 Let $M$ be a properly immersed $\lambda-$hypersurface in $\mathbb{R}^{n+1}$.    Then the drifted Laplacian $\mathcal{L}$  with domain $L^2(M,e^{-\frac{|x|^2}{4}}d\sigma)$ has discrete spectrum.
\end{proposition}
These results are corollaries of the following much more general results. 

\begin{theorem}\label{tpdis} Let $M$ be a properly immersed submanifold in a Ricci shrinker $(\overline{M},g,f)$ with $\overline{Ric}_f=kg$. If there are nonnegative constants $\epsilon$, $\rho$ and $C$ with $\epsilon+\rho<2k$ such that the trace of Hessian of $f$ on $\overline{M}$ restricted on the normal bundle of $M$ satisfies $\overline{R}+2\textrm{tr}_{M^{\perp}} \overline{\nabla}^2f\geq - \epsilon  f$ outside a compact set and the norm of weighted mean curvature vector $\vec{H}_f$ of $M$ is bounded above by $\sqrt{\rho f+C}$, then the spectrum of $\Delta_f$ with domain $L^2(M,\mu)$ is discrete.
\end{theorem}

\begin{proof} Let $U:L^2(M)\rightarrow L^2(M,\mu)$ be  the unitary isomorphism  given by $U(u)=ue^{\frac{f}{2}}$ and
the operator $L=\Delta+\frac{1}{2}\Delta f-\frac{1}{4}|\nabla f|^2$. A direct computation shows that $\Delta_f=ULU^{-1}$.  It follows that the operator $\Delta_f$ with domain $L^2(M,\mu)$ has a discrete spectrum if and only if the operator $L$ with domain $L^2(M)$ has a discrete spectrum. Using spectral theory to prove that $L$ has a discrete spectrum, it suffices to show that the function $\frac{1}{4}|\nabla f|^2-\frac{1}{2}\Delta f$ goes to infinity when $x\to \infty$ (e.g. \cite{RS1975} p. 120). By equation (\ref{eq}), we have
$$\Delta f=\overline{\Delta}f+\langle \vec{H}_f,(\overline{\nabla}f)^{\perp}\rangle - |(\overline{\nabla}f)^{\perp}|^2 - \textrm{tr}_{M^{\perp}} \overline{\nabla}^2f.$$

This implies that
$$|\nabla f|^2-2\Delta f =|\overline{\nabla}f|^2 -2\overline{\Delta}f-2\langle \vec{H}_f,(\overline{\nabla}f)^{\perp}\rangle +|(\overline{\nabla}f)^{\perp}|^2 +2\textrm{tr}_{M^{\perp}} \overline{\nabla}^2f.$$

Since $\overline{\textrm{Ric}}_f=kg$ we obtain $\overline{\Delta} f+\overline{R}=km$, where $m=\dim \overline{M}$. Furthermore, by Theorem \ref{hamilton}, there is a constant $\lambda$ such that $|\overline{\nabla}f|^2+\overline{R}=2kf+\lambda$. It follows that
\begin{eqnarray*}
|\nabla f|^2-2\Delta f &=& 2kf+\overline{R}-2km+\lambda-2\langle \vec{H}_f,(\overline{\nabla}f)^{\perp}\rangle +|(\overline{\nabla}f)^{\perp}|^2\\
& & +2\textrm{tr}_{M^{\perp}} \overline{\nabla}^2f\\
&=& 2kf+\overline{R}-2km+\lambda+|\vec{H}_f-(\overline{\nabla}f)^{\perp}|^2 -|\vec{H}_f|^2 \\
& & +2\textrm{tr}_{M^{\perp}} \overline{\nabla}^2f.\\
&=& 2kf-2km+\lambda-|\vec{H}_f|^2 +|\vec{H}|^2 +\overline{R}+2\textrm{tr}_{M^{\perp}} \overline{\nabla}^2f.
\end{eqnarray*}

By hypothesis, we have 
\begin{equation}\label{ineq1}
\frac{1}{4}|\nabla f|^2-\frac{1}{2}\Delta f \geq  \frac{(2k-\rho-\epsilon)f}{4}-\frac{(2km+C-\lambda)}{4}  
\end{equation}
\noindent outside a compact set. By Remark \ref{r1}, $f$ goes to infinity when $x\to \infty$ on $\overline{M}$. So, since $M$ is properly immersed in $\overline{M}$, it follows that the restriction of $f$ to $M$ goes to infinity when $x\to \infty$ on $M$. This and inequality (\ref{ineq1}) prove the result.
\end{proof}

\section{An integral inequality of the weighted Bochner formula}\label{sec4}

 In this section, we will prove an integral inequality of the weighted Bochner formula which we will apply in the following section to prove main Theorem \ref{tp}. This is similar to Reilly's formula for compact domains on Riemannian manifolds.

 Let $(\overline{M},g,f)$ be a complete smooth metric measure space and  $M$ be a properly embedded hypersurface in $\overline{M}$. Consider a domain $\Omega\subset\overline{M}$ such that $M=\partial \Omega$ and denote by $\eta$ the outward-pointing unit vector normal to $\partial\Omega$. Denote by $B$ the second fundamental form of $M$ with respect to $\eta$.

\begin{lemma}\label{prop1} If $\Phi \in H^1(\overline{M}, \nu)$ such that the restriction of $\Phi$ to $\Omega$ is smooth up to $\partial\Omega$ then 
\begin{align*}
  & \int_{\Omega}\psi^2\left[|\overline{\nabla}^2\Phi|^2+ 2\overline{\textrm{Ric}}_f(\overline{\nabla}\Phi,\overline{\nabla}\Phi) - 3\left(\bar{\Delta}_f\Phi\right)^2\right] d\nu & \\ 
  & \leq  8\int_{ \Omega}|\overline{\nabla}\Phi|^2|\overline{\nabla}\psi|^2 d\nu -4\int_M u\psi \langle \nabla\Phi,\nabla\psi\rangle d\mu & \\
  & \quad + 2\int_M \psi^2 B(\nabla\Phi,\nabla\Phi) d\mu+2\int_M \psi^2u\langle\vec{H}_f, \overline{\nabla}\Phi\rangle d\mu-4\int_M \psi^2u\Delta_f\Phi \ d\mu & \\
\end{align*}\noindent for every  $\psi\in C^{\infty}_0(\overline{M})$, where $u=\left.\langle \overline{\nabla}\Phi, \eta\rangle\right|_M$.
\end{lemma}

\begin{proof}
Consider $\psi\in C^{\infty}_0(\overline{M})$. Since $M$ is a properly embedded hypersurface in $\overline{M}$ then the restriction of $\psi$ to $M$ belongs to $C^{\infty}_0(M)$. By the weighted Green formula and the weighted Bochner formula:
$$\overline{\textrm{div}}_f\left(\frac{1}{2}\overline{\nabla}|\overline{\nabla}\Phi|^2-\overline{\Delta}_f\Phi\overline{\nabla}\Phi\right)=|\overline{\nabla}^2\Phi|^2+\overline{\textrm{Ric}}_f(\overline{\nabla}\Phi,\overline{\nabla}\Phi)-\left(\overline{\Delta}_f\Phi\right)^2.$$

\noindent we have
\begin{eqnarray*}
X & := & \int_{\Omega}\psi^2\left[|\overline{\nabla}^2\Phi|^2+ \overline{\textrm{Ric}}_f(\overline{\nabla}\Phi,\overline{\nabla}\Phi) - \left(\bar{\Delta}_f\Phi\right)^2\right] d\nu \\
& = &\int_{\Omega}\psi^2\overline{\textrm{div}}_f\left(\frac{1}{2}\overline{\nabla}|\overline{\nabla}\Phi|^2-\overline{\Delta}_f\Phi\overline{\nabla}\Phi\right)d\nu \\
& = & -\int_{\Omega}\left\langle \overline{\nabla}\psi^2, \frac{1}{2}\overline{\nabla}|\overline{\nabla}\Phi|^2-\overline{\Delta}_f\Phi\overline{\nabla}\Phi\right\rangle d\nu \\
&  & + \int_M \psi^2\left\langle \eta,\frac{1}{2}\overline{\nabla}|\overline{\nabla}\Phi|^2-\overline{\Delta}_f\Phi\overline{\nabla}\Phi\right\rangle d\mu  \\
& = & -\int_{\Omega}\psi\left\langle \overline{\nabla}\psi, \overline{\nabla}|\overline{\nabla}\Phi|^2\right\rangle d\nu +2\int_{\Omega}\psi\overline{\Delta}_f\Phi\left\langle \overline{\nabla}\psi,\overline{\nabla}\Phi\right\rangle d\nu \\
&  & + \int_M \psi^2\left(\frac{1}{2}\langle \eta,\overline{\nabla}|\overline{\nabla}\Phi|^2\rangle-u\overline{\Delta}_f\Phi\right) d\mu  \\
& = & -2\int_{\Omega}\psi\overline{\nabla}^2\Phi\left(\overline{\nabla}\psi, \overline{\nabla}\Phi\right) d\nu+2\int_{\Omega}\psi\overline{\Delta}_f\Phi\left\langle \overline{\nabla}\psi,\overline{\nabla}\Phi\right\rangle d\nu \\
&  & + \int_M \psi^2\left(\frac{1}{2}\langle \eta,\overline{\nabla}|\overline{\nabla}\Phi|^2\rangle-u\overline{\Delta}_f\Phi\right) d\mu\\
\end{eqnarray*}

By Young’s inequality $2ab \leq 2a^2 + \frac{1}{2}b^2$, we have
\begin{eqnarray}
X & \leq & 4\int_{\Omega}|\overline{\nabla}\Phi|^2|\overline{\nabla}\psi|^2 d\nu+\frac{1}{2}\int_{\Omega}\psi^2 |\overline{\nabla}^2\Phi|^2d\nu+ \frac{1}{2}\int_{\Omega}\psi^2 (\overline{\Delta}_f\Phi)^2 d\nu \nonumber \\
\label{e3} & + & \int_M \psi^2\left(\frac{1}{2}\langle \eta,\overline{\nabla}|\overline{\nabla}\Phi|^2\rangle-u\overline{\Delta}_f\Phi\right)d\mu
\end{eqnarray}

On $M$, we have
$$\overline{\nabla}\Phi=\nabla \Phi+u\eta=\sum_{i=1}^n\nabla_i\Phi e_i+u\eta$$
\noindent where $\{e_1, \cdots , e_n\}$ is a local orthonormal frame on $M$. By (\ref{eq1}), on $M$, we obtain
\begin{eqnarray*}
\frac{1}{2}\langle \eta,\overline{\nabla}|\overline{\nabla}\Phi|^2\rangle-u\overline{\Delta}_f\Phi & = & \overline{\nabla}^2\Phi(\eta,\overline{\nabla}\Phi)- u\Delta_f\Phi+u\langle\vec{H}_f, \overline{\nabla}\Phi\rangle\\
& & -u\overline{\nabla}^2\Phi(\eta,\eta)\\
& = & \sum_{i=1}^n\nabla_i\Phi \overline{\nabla}^2\Phi(\eta,e_i)- u\Delta_f\Phi +u\langle\vec{H}_f, \overline{\nabla}\Phi\rangle\\
\end{eqnarray*}

Note that
\begin{eqnarray*}
 \overline{\nabla}^2\Phi(\eta,e_i) & = & \left\langle \overline{\nabla}_{e_i}\overline{\nabla}\Phi,\eta\right\rangle\\
 &=& e_i\langle \overline{\nabla}\Phi,\eta\rangle- \langle \overline{\nabla}\Phi, \overline{\nabla}_{e_i}\eta\rangle\\
 &=& \nabla_iu- \langle \nabla \Phi, \overline{\nabla}_{e_i}\eta\rangle\\
 &=& \nabla_iu- \sum_{j=1}^n \nabla_j\Phi \langle e_j, \overline{\nabla}_{e_i}\eta\rangle\\
 &=& \nabla_iu+\sum_{j=1}^n \nabla_j\Phi B_{ij}\\
\end{eqnarray*}

So, on $M$, we obtain
\begin{eqnarray}
\frac{1}{2}\langle \eta,\overline{\nabla}|\overline{\nabla}\Phi|^2\rangle-u\overline{\Delta}_f\Phi & = & \sum_{i=1}^n\nabla_i\Phi\nabla_iu+ \sum_{i,j=1}^n\nabla_i\Phi\nabla_j\Phi B_{ij}- u\Delta_f\Phi \nonumber \\
&  & + u\langle\vec{H}_f, \overline{\nabla}\Phi\rangle \nonumber \\
\label{e4}& = & \langle \nabla\Phi,\nabla u\rangle + B(\nabla\Phi,\nabla\Phi)- u\Delta_f\Phi\\
& &  +u\langle\vec{H}_f, \overline{\nabla}\Phi\rangle.\nonumber
\end{eqnarray}

It follows from (\ref{e3}) and (\ref{e4}) that
\begin{align*}
 & \int_{\Omega}\psi^2\left[|\overline{\nabla}^2\Phi|^2+ 2\overline{\textrm{Ric}}_f(\overline{\nabla}\Phi,\overline{\nabla}\Phi) - 3\left(\bar{\Delta}_f\Phi\right)^2\right] d\nu & \\ 
& \leq  8\int_{\Omega}|\overline{\nabla}\Phi|^2|\overline{\nabla}\psi|^2 d\nu+ 2\int_M \psi^2\langle \nabla\Phi,\nabla u\rangle d\mu & \\
& + 2\int_M \psi^2 B(\nabla\Phi,\nabla\Phi) d\mu-2\int_M \psi^2u\Delta_f\Phi \ d\mu + 2\int_M \psi^2u\langle\vec{H}_f, \overline{\nabla}\Phi\rangle d\mu
\end{align*}

By the weighted Green formula, we have
\begin{eqnarray*}\int_M \psi^2\langle \nabla\Phi,\nabla u\rangle d\mu & = & \int_M \langle \nabla\Phi,\nabla (u\psi^2)\rangle d\mu-2\int_M u\psi \langle \nabla\Phi,\nabla\psi\rangle d\mu\\
& = & -\int_M \psi^2 u \Delta_f\Phi \ d\mu-2\int_M u\psi \langle \nabla\Phi,\nabla\psi\rangle d\mu\\
\end{eqnarray*}

Therefore
\begin{align*}
  & \int_{\Omega}\psi^2\left[|\overline{\nabla}^2\Phi|^2+ 2\overline{\textrm{Ric}}_f(\overline{\nabla}\Phi,\overline{\nabla}\Phi) - 3\left(\bar{\Delta}_f\Phi\right)^2\right] d\nu & \\ 
  & \leq  8\int_{ \Omega}|\overline{\nabla}\Phi|^2|\overline{\nabla}\psi|^2 d\nu -4\int_M u\psi \langle \nabla\Phi,\nabla\psi\rangle d\mu & \\
  & + 2\int_M \psi^2 B(\nabla\Phi,\nabla\Phi) d\mu+2\int_M \psi^2u\langle\vec{H}_f, \overline{\nabla}\Phi\rangle d\mu-4\int_M \psi^2u\Delta_f\Phi \ d\mu & \\
\end{align*}
\end{proof}

Lemma \ref{prop1} has the following corollary.

\begin{cor}\label{cor1}  Assume that $(\overline{M},g,f)$ is a Ricci shrinker with $\overline{\textrm{Ric}}_f=\frac{1}{2} g$ and $M$ is a properly embedded $f$-minimal hypersurface in $\overline{M}$. Also assume that $\Phi \in H^1(\overline{M}, \nu)$ satisfies the following conditions:
\begin{itemize}
\item  The restriction of $\Phi$ to $\Omega$ is smooth up to $\partial\Omega$;
\item $\overline{\Delta}_f\Phi=0$ on $\Omega$;
\item  The restriction of $\Phi$ to $M$ is a nonconstant eigenfunction of $\Delta_f$ corresponding to an eigenvalue $\lambda$, i.e., $\Delta_f \Phi +\lambda \Phi=0.$
\end{itemize}
Then 
\[ \begin{split}
    \int_{\Omega}\psi^2|\overline{\nabla}^2\Phi|^2 d\nu+ & \left(1-4\lambda\right)\int_{\Omega}\psi^2|\overline{\nabla} \Phi|^2 d\nu\\
    & \leq 8\int_{\Omega}|\overline{\nabla}\Phi|^2|\overline{\nabla}\psi|^2 d\nu-4\int_M u\psi \langle \nabla\Phi,\nabla\psi\rangle d\mu\\ 
    &\quad+ 2\int_M \psi^2 B(\nabla\Phi,\nabla\Phi) d\mu+8\lambda\int_{\Omega} \psi \Phi\langle \overline{\nabla}\Phi, \overline{\nabla}\psi\rangle d\nu.
\end{split}
\] \noindent for every $\psi\in C^{\infty}_0(\overline{M})$.
\end{cor}
\begin{proof} By Proposition \ref{prop1}, we have

\[\begin{split}
\int_{\Omega}\psi^2|\overline{\nabla}^2\Phi|^2 d\nu & + \int_{\Omega}\psi^2|\overline{\nabla} \Phi|^2 d\nu \\
& \leq  8\int_{\Omega}|\overline{\nabla}\Phi|^2|\overline{\nabla}\psi|^2 d\nu-4\int_M u\psi \langle \nabla\Phi,\nabla\psi\rangle d\mu\\
& \quad+ 2\int_M \psi^2 B(\nabla\Phi,\nabla\Phi) d\mu+4\lambda\int_M \psi^2 u \Phi d\mu.
\end{split}\]

Since $\overline{\Delta}_f\Phi=0$ on $\Omega$, by Green formula, we obtain
\begin{eqnarray*}
4\lambda\int_M \psi^2 u \Phi d\mu & = & 4\lambda\int_M \psi^2\Phi \langle \overline{\nabla} \Phi, \eta\rangle d\mu\\
& = & 4\lambda\int_{\Omega} \langle \overline{\nabla}\Phi, \overline{\nabla}(\psi^2\Phi)\rangle d\nu\\
& = & 4\lambda\int_{\Omega} \psi^2 |\overline{\nabla}\Phi|^2 d\nu +8\lambda\int_{\Omega} \psi \Phi\langle \overline{\nabla}\Phi, \overline{\nabla}\psi\rangle d\nu
\end{eqnarray*}

This prove the result.

\end{proof}

\section{Lower bounded for $\lambda_1(\Delta_f)$}\label{sec5}

\medskip
In this section we will give the proof of the estimate of the first nonzero eigenvalue for the drift Laplacian for properly embedded $f$-minimal hypersurfaces.

\begin{theorem} Let $M$ be a $f$-minimal properly embedded hypersurface in a Ricci shrinker $(\overline{M},g,f)$ such that $\overline{\textrm{Ric}}_f=k g$. Assume that there is a constant $0\leq\epsilon<2k$ such that the trace of Hessian of $f$ on $\overline{M}$ restricted on the normal bundle of $M$ satisfies $\overline{R}+2\textrm{tr}_{M^{\perp}} \overline{\nabla}^2f\geq - \epsilon  f$ outside a compact set. Then the spectrum of $\Delta_f$ with domain on $L^2(M,\mu)$ is discrete and the first nonzero eigenvalue $\lambda_1(\Delta_f)$ satisfies
$$\lambda_1(\Delta_f)\geq \frac{k}{2}.$$
\end{theorem}


\begin{proof} By a scaling of the metric $g$, we can assume that $k=\frac{1}{2}$. In this case, we have $(\overline{M},g,f)$ is a Ricci shrinker with $\overline{\textrm{Ric}}_f=\frac{1}{2}g$. By Remark \ref{r2} we have $\overline{M}$ has finite fundamental group. We first assume that $\overline{M}$ is simply connected. Since $M$ is embedded in $\overline{M}$ then $M$ is orientable and the complement $\overline{M}\setminus M$ has two connected components $\Omega$ and $\tilde{\Omega}$ such that $\partial\Omega=\partial \tilde{\Omega}=M$ (see \cite{CS1985}). We denote by $\eta$ the outward-pointing unit normal to $\partial\Omega$ and by $B$ the second fundamental form of $M$ with respect to $\eta$.

The discreteness of spectrum of $\Delta_f$ follows from Theorem \ref{tpdis} and guarantees the existence of a nonconstant eigenfunction $\varphi\in H^1(M,\mu)\cap C^{\infty}(M)$ of $\Delta_f$ corresponding to an eigenvalue $\lambda$, i.e,
$$\Delta_f\varphi+\lambda\varphi=0.$$

We claim that there is a function $\Phi\in H^1(\overline{M},\nu)$ such that $\left.\Phi\right|_M=\varphi$, $\overline{\Delta}_f\Phi=0$ on $\Omega$ and $\overline{\Delta}_f\Phi=0$ on $\tilde{\Omega}$. To prove this statement, define the functional $E:H^1(\overline{M},\nu)\rightarrow \mathbb{R}$ by
$$E(\Phi)=\int_{\overline{M}}|\overline{\nabla}\Phi|^2 d\nu\geq 0.$$

Denote by $\alpha:=\displaystyle\inf_{\mathcal{A}} E$, where $$\mathcal{A}:=\{\Phi\in H^1(\overline{M},\nu); \ \left.\Phi\right|_M=\varphi\}.$$ Consider $(\Phi_k)_{k\in\mathbb{N}}\subset \mathcal{A}$ a sequence such that $E(\Phi_k)\rightarrow \alpha.$ We have
\begin{equation}\label{e1a}
\sup_{k\in\mathbb{N}}\int_{\overline{M}}|\overline{\nabla}\Phi_k|^2 d\nu<\infty.\end{equation}

Since $(\overline{M},g,f)$ is a Ricci shrinker with $\overline{\textrm{Ric}}_f=\frac{1}{2}g$, holds the Poincaré Inequality:
$$\int_{\overline{M}}(\Phi-\bar\Phi)^2d\nu\leq 2 \int_{\overline{M}}|\overline{\nabla}\Phi|^2d\nu$$
\noindent for every $\Phi\in H^1(\overline{M},\nu)$, where 
$$\bar\Phi=\frac{1}{\nu(\overline{M})}\int_{\overline{M}}\Phi d\nu.$$

So, by (\ref{e1a}) we have
\begin{equation}\label{e2a}\sup_{k\in\mathbb{N}}\int_{\overline{M}}(\Phi_k-\bar\Phi_k)^2d\nu<\infty.\end{equation}

Fix a point $p\in M$ and consider $B_r$ the geodesic ball in $\overline{M}$ of radius $r>0$ centered at $p$. By (\ref{e1a}), (\ref{e2a}) and the Sobolev trace theorem, we have there is a nonnegative constant $C$ depending of $r$, such that
$$\int_{M\cap B_r}(\varphi-\bar\Phi_k)^2d\mu\leq C$$
\noindent for every $k\in\mathbb{N}$. This implies that $\displaystyle\sup_{k\in\mathbb{N}}(\bar\Phi_k)^2<\infty$. It follows from (\ref{e2a}) that 
$$\sup_{k\in\mathbb{N}}\int_{\overline{M}}\Phi_k^2d\nu<\infty.$$

It follows that $(\Phi_k)_{k\in\mathbb{N}}$ is a bounded sequence in $H^1(\overline{M},\nu)$. Since $(\overline{M},g,f)$ is a Ricci shrinker then from Corollary 1 in \cite{XCDZ} $H^1(\overline{M},\nu)\subset\subset L^2(\overline{M},\nu)$. By passing a subsequence, we have that there is $\Phi\in H^1(\overline{M},\nu)$ such that $\Phi_k\rightarrow \Phi$ in $H^1(\overline{M},\nu)$. In particular, we conclude that $\Phi\in \mathcal{A}$. Using the continuity of $E$  we obtain $E(\Phi)=\alpha$. Therefore, the statement follows. 

If $\overline{M}$ is compact, by Corollary \ref{cor1}, we have
 $$\int_{\Omega}|\overline{\nabla}^2\Phi|^2 d\nu+\left(1-4\lambda\right)\int_{\Omega}|\overline{\nabla} \Phi|^2 d\nu \leq 2\int_M B(\nabla\varphi,\nabla\varphi) d\mu.$$
 We can assume that $\displaystyle\int_M B(\nabla\varphi,\nabla\varphi) d\mu\leq 0$, otherwise we can work with $\tilde{\Omega}$ rather than with $\Omega$. So, in this case, we obtain
 \begin{equation}\label{eqtp}
 \int_{\Omega}|\overline{\nabla}^2\Phi|^2 d\nu+\left(1-4\lambda\right)\int_{\Omega}|\overline{\nabla} \Phi|^2 d\nu \leq 0.
 \end{equation}

If $\overline{M}$ is noncompact, we claim that (\ref{eqtp}) also holds. To prove this statement, fix a point $p\in M$. For each $k\in \mathbb{N}$, consider $B_k$ the geodesic ball in $\overline{M}$ of radius $k$ centered at $p$. Consider a nonnegative function $\psi_k\in C_0^{\infty}(\overline{M})$ such that $\psi_k=1$ on $B_k$, $|\overline{\nabla} \psi_k|\leq 1$ on $B_{k+1}\setminus B_k$ and $\psi_k=0$ on $\overline{M}\setminus B_{k+1}$. By Corollary \ref{cor1}, we have

\[\begin{split}
    \int_{\Omega}\psi_k^2|\overline{\nabla}^2\Phi|^2 d\nu+&\left(1-4\lambda\right)\int_{\Omega}\psi_k^2|\overline{\nabla} \Phi|^2 d\nu \\
    & \leq 8\int_{\Omega}|\overline{\nabla}\Phi|^2|\overline{\nabla}\psi_k|^2 d\nu-4\int_M u\psi_k \langle \nabla\varphi,\nabla\psi_k\rangle d\mu\\
& \quad+ 2\int_M \psi_k^2 B(\nabla\varphi,\nabla\varphi) d\mu+8\lambda\int_{\Omega} \psi_k \Phi\langle \overline{\nabla}\Phi, \overline{\nabla}\psi_k\rangle d\nu.
\end{split}
\]
\noindent where $u=\left.\langle \overline{\nabla}\Phi, \eta\rangle\right|_M$. We can assume $$\left\{k\in\mathbb{N}; \ \int_M \psi_k^2 B(\nabla\varphi,\nabla\varphi) d\mu\leq 0\right\}$$
\noindent is a infinite set, otherwise we can work with $\tilde{\Omega}$ rather than with $\Omega$. So, by passing to a subsequence, we have
\[\begin{split}
\int_{\Omega}\psi_k^2|\overline{\nabla}^2\Phi|^2 d\nu+ & \left(1-4\lambda\right)\int_{\Omega}\psi_k^2|\overline{\nabla} \Phi|^2 d\nu \\
& \leq 8\int_{\Omega}|\overline{\nabla}\Phi|^2|\overline{\nabla}\psi_k|^2 d\nu-4\int_M u\psi_k \langle \nabla\varphi,\nabla\psi_k\rangle d\mu\\
& \label{1} \quad+ 8\lambda\int_{\Omega} \psi_k \Phi\langle \overline{\nabla}\Phi, \overline{\nabla}\psi_k\rangle d\nu.
\end{split}\]

It follows from the Cauchy-Schwarz inequality that
\[\begin{split}
\int_{\Omega}\psi_k^2|\overline{\nabla}^2\Phi|^2 d\nu+&\left(1-4\lambda\right)\int_{\Omega}\psi_k^2|\overline{\nabla} \Phi|^2 d\nu \\
& \leq 8\int_{\Omega}|\overline{\nabla}\Phi|^2|\overline{\nabla}\psi_k|^2 d\nu+4\int_M |u| |\psi_k| |\nabla\varphi| |\nabla\psi_k| d\mu\\
& \quad + 8\lambda\int_{\Omega} |\psi_k| |\Phi| |\overline{\nabla}\Phi| |\overline{\nabla}\psi_k| d\nu.
\end{split}\]

Note that, by the Dominated Convergence Theorem, the right-hand side of the inequality above goes to $0$ when $k\rightarrow \infty$. However, by the Monotone Convergence Theorem, we have
$$\lim_{k \to \infty} \int_{\Omega}\psi_k^2|\overline{\nabla}^2\Phi|^2 d\nu=\int_{\Omega}|\overline{\nabla}^2\Phi|^2 d\nu$$
\noindent and
$$\lim_{k \to \infty} \int_{\Omega}\psi_k^2|\overline{\nabla} \Phi|^2 d\nu=\int_{\Omega}|\overline{\nabla} \Phi|^2 d\nu.$$

This implies that
$$\int_{\Omega}|\overline{\nabla}^2\Phi|^2 d\nu+\left(1-4\lambda\right)\int_{\Omega}|\overline{\nabla} \Phi|^2 d\nu\leq 0.$$
So the claim holds. Since $\Phi$ is a nonconstant function, we have $\lambda\geq \frac{1}{4}$. Therefore, the result is proven if $\overline{M}$ is simply connected.

Assume now that $\overline{M}$ is not simply connected. Since $\overline{M}$ has finite fundamental group, its universal covering $\pi: \hat{M} \rightarrow \overline{M}$ has a finite number of sheets. In particular, $\pi$ is a proper function. Consider the Riemannian metric $\hat{g}$ on $\hat{M}$ given by the pullback of $g$ by $\pi$. It is well know that $\pi: (\hat{M},\hat{g}) \rightarrow (\overline{M},g)$ is a local isometry. So, taking $\hat{f}=f\circ \pi$, we have $(\hat{M},\hat{g},\hat{f})$ is a Ricci shrinker with $\hat{R}ic_{\hat{f}}=\frac{1}{2}\hat{g}$. Furthermore, the lift $\Sigma$ of $M$ is a $\hat{f}$-minimal properly embedded hypersurface in $\hat{M}$ such that the trace of Hessian of $\hat{f}$ on $\hat{M}$ restricted on the normal bundle of $\Sigma$ satisfies $\hat{R}+2\textrm{tr}_{\Sigma^{\perp}} \hat{\nabla}^2\hat{f}\geq - \epsilon  \hat{f}$ outside a compact set. Since $\hat{M}$ is simply connected, we can apply this result, which has already been proven when the ambient manifold is simply connected. So, we obtain the first eigenvalue of the drifted Laplacian $\Delta_{\hat{f}}$ on $\Sigma$ satisfies $\lambda_1(\Delta_{\hat{f}})\geq \frac{1}{4}$. Hence,
$$\lambda_1(\Delta_f)\geq \lambda_1(\Delta_{\hat{f}})\geq \frac{1}{4}.$$

\end{proof}


\end{document}